\documentclass[12pt,a4paper,draft]{article}
 \usepackage{latexsym}
 \usepackage[leqno, namelimits, sumlimits]{amsmath}
 \usepackage{amssymb, amsthm, amsfonts}
\newtheorem{Theorem}{Theorem}[section]
\newtheorem{Lemma}[Theorem]{Lemma}
\newtheorem{Corollary}[Theorem]{Corollary}

\newtheorem{proposition}[Theorem]{Proposition}

\newtheorem{remark}[Theorem]{Remark}

\newcommand{\R}{{\mathbb R}}
\newcommand{\N}{{\mathbb N}}

\newcommand{\eps}{\varepsilon}
\def\({\Bigl(}
\def\){\Bigr)}

\newcommand{\al}{\alpha}
\newcommand{\be}{\beta}
\newcommand{\ga}{\gamma}

\newcommand{\Om}{\Omega}

\newcommand{\pa}{\partial}

\newcommand{\cT}{{\mathcal T}}
\newcommand{\cS}{{\mathcal S}}
\newcommand{\cP}{{\mathcal P}}
\DeclareMathOperator{\opspan}{span}
\DeclareMathOperator{\opkern}{kern}
\DeclareMathOperator{\oploc}{loc}
\DeclareMathOperator{\diag}{diag}
\numberwithin{equation}{section}

\title{Bifurcation in a multi-component system of nonlinear Schr\"odinger equations}
\author{Thomas Bartsch}

\date{}

\begin{document}

\maketitle

\begin{center}
{\sl Dedicated to Professor Kasimierz G\c eba.}
\end{center}

\vspace{.5cm}
\begin{abstract}
We consider the system
\begin{equation*}
-\Delta u_j + a(x)u_j = \mu_j u_j^3 + \be\sum_{k\ne j}u_k^2u_j,\ u_j>0,
  \qquad  j=1,\dots,n,
\end{equation*}
on a possibly unbounded domain $\Om\subset\R^N$, $N\le3$, with Dirichlet boundary conditions. The system appears in nonlinear optics and in the analysis of mixtures of Bose-Einstein condensates. We consider the self-focussing (attractive self-interaction) case $\mu_1,\dots,\mu_n > 0$ and take $\be\in\R$ as bifurcation parameter. There exists a branch of positive solutions with $u_j/u_k$ being constant for all $j,k\in\{1,\dots,n\}$. The main results are concerned with the bifurcation of solutions from this branch. Using a hidden symmetry we are able to prove global bifurcation even when the linearization has even-dimensional kernel (which is always the case when $n>1$ is odd).
\end{abstract}



%%%%%%%%%%%%%%%%%%%%%%%%%%%%%%%%%%%%%%%%%%%%%%%%%%%%%%%%%%%%%%%%%%%%%%%%%%%%%%%%%%%

\section{Introduction}
\label{intro}
The system of coupled Gross-Pitaevskii equations
\begin{equation} \label{eq:NLS}
-i\partial_t\psi_j(x,t)
    = \Delta_x \psi_j + \mu_j |\psi_j|^2\psi_j + \be\sum_{k\ne j}|\psi_k|^2\psi_j
  \qquad  j=1,\dots,n,
\end{equation}
in $\R^N$, $N\le3$, with parameters $\mu_1,\dots,\mu_n>0$, $\be\in\R$, has found
considerable interest in the last years. It appears in nonlinear optics and in
models for mixtures of Bose-Einstein condensates; see
\cite{AA,esry-etal:1997,malomed:2008} for physics, and
\cite{amco2,bartsch-dancer-wang:2010,bartsch-wang:2006,bartsch-wang-wei:2007,
lin-wei:2005a,MMP,nttv,pom,sirakov:2007} for mathematics papers. We deal with the case of attractive self-interaction, i.~e.\ $\mu_1,\dots,\mu_n>0$. Concerning $\be\in\R$ our bifurcating solutions appear in the range $\be<\mu_1$; in fact, all but finitely many bifurcations appear in the range $\be<0$ of repelling interaction between different components. The ansatz
\[
\psi_j(x,t) = e^{it}u_j(x),  \qquad  j=1,\dots,n,
\]
for stationary waves leads to the following elliptic system for the amplitudes
$u_1,\dots,u_n$:
\begin{equation} \label{eq:special}
\left\{
 \begin{aligned}
  &-\Delta u_j + u_j = \mu_j u_j^3 + \be\sum_{k\neq j}u_k^2u_j,\\
  &u_j\in H^1(\R^N),\ u_j>0,  \\
  \end{aligned}
 \right.\qquad j=1,\dots,n.
\end{equation}
We keep $\mu_1,\ldots,\mu_n>0$ fixed and take $\be$ as bifurcation parameter.

Most of the above mentioned papers papers deal with the case $n=2$ of two components. Using variational methods, the existence of ground and bound states is obtained. In \cite{bartsch-dancer-wang:2010} a different approach is pursued using bifurcation methods, but only for $n=2$. We shall extend this bifurcation approach from $n=2$ to arbitrary $n\ge2$. This leads to interesting new features and difficulties. For instance, in the radial setting of \cite{bartsch-dancer-wang:2010} the linearizations have a one-dimensional kernel at the bifurcation parameter. This allows the use of degree theory to prove global bifurcation of solutions.

For more than two equations we have to deal with high-dimensional kernels forced by the structure of the system. In fact, in the radial setting the dimension of the kernels is precisely $n-1$. In the general case, the dimensions of the kernels are multiples of $n-1$. Therefore if $n$ is odd the kernels always have even dimensions, so there will never be a change of degree to prove bifurcation. Observe that there is no symmetry subgroup $G$ of the symmetric group $S_n$ leaving the system invariant, since we do not assume that (some of) the coefficients $\mu_1,\ldots,\mu_n$ are equal.

We discover a hidden symmetry of the problem which explains the high nullities at the bifurcation parameters, and which can be used to prove multiple global bifurcation branches. Observe that due to the lack of a manifest symmetry of the system (except for the radial symmetry of the domain which allows us to work on the space of radially symmetric functions), we cannot apply the equivariant degree or the equivariant gradient degree as presented in \cite{Balanov-Krawcewicz-Steinlein:2006,Geba:1997,Rybicki:2005} and the references therein.

Actually, we shall deal with a more general problem. Let $\Om\subset\R^N$, $N\le3$, be a domain which is invariant under a closed subgroup $G\subset O(N)$, and suppose $a\in L^\infty_{\oploc}(\Om)$ is invariant under $G$. We require that $-\Delta+a$ is a positive self-adjoint operator on $L^2(\Om)$, and define
\[
  E:=\left\{u\in H^1_0(\Om): u\text{ is }G\text{-invariant}, \int_\Om au^2<\infty\right\}.
\]
The norm in $E$ is given by
\[
  \|u\|^2_E=\int_\Om|\nabla u|^2+\int_\Om a u^2.
\]
We also require that
\begin{itemize}
\item[(A)] the embedding $E\hookrightarrow L^4(\Om)$ is compact.
\end{itemize}
This is the case, for instance, if $\Om$ is bounded and $G$ is trivial, or if $\Om$ is radially symmetric and $G=O(N)$, $2\le N\le3$. It is also the case if $\Om=\R^N$, $2\le N\le3$, and $a(x)\to\infty$ as $|x|\to\infty$; see \cite{bartsch-pankov-wang:2001} for more general results in this direction. The problem we investigate is
\begin{equation}
  \label{eq:main}
  \left\{
 \begin{aligned}
  &-\Delta u_j + a(x)u_j = \mu_j u_j^3 + \be\sum_{k\neq j}u_k^2u_j\ ,
      &&\qquad j=1,\dots,n,\\
  &u_j\in E,\ u_j>0.  \\
  \end{aligned}
 \right.
\end{equation}
As mentioned above, this covers the important special case $\Om=\R^N$, $N=2$ or $N=3$, $E=H^1_{rad}(\R^N)$.

A solution $(u_1,\ldots,u_n)$ of \eqref{eq:main} is said to be locked if $u_i/u_j$ is constant for all $i,j$. In this paper we first describe the set of locked solutions of \eqref{eq:main}. We then investigate the bifurcation of non-locked solutions from the set of locked solutions. In particular, we shall prove the bifurcation of partially locked solutions where $u_i/u_j$ is constant for some $i\ne j$, but not for all indices.

Without loss of generality we assume $0<\mu_1\leq\mu_2\leq\ldots\leq\mu_n$ throughout the paper.

%%%%%%%%%%%%%%%%%%%%%%%%%%%%%%%%%%%%%%%%%%%%%%%%%%%%%%%%%%%%%%%%%%%%%%%%%%%%%%%%%%%

\section{Branches of locked solutions}
\label{sec:locked branches}

Given a solution of the scalar equation
\begin{equation}\label{eq:scalar}
\left\{
    \begin{aligned}
{}&      -\Delta w+a(x)w=w^3\\
{}&w\in E, w>0
    \end{aligned}
\right.
\end{equation}
there exists a branch $\cT_w\subset\R\times E^n$ of locked solutions $(\be,u_1,\ldots,u_n)$ of \eqref{eq:main} such that each $u_i$ is a multiple of $w$. In order to describe this branch we set
\begin{equation}\label{eq:def-g}
  g(\be):=1+\be\sum^n_{k=1}\frac{1}{\mu_k-\be}
\end{equation}
and for $j=1,\ldots,n$:
\begin{equation}\label{eq:def-al}
\al_j(\be):=\big((\mu_j-\be)g(\be)\big)^{-1/2}
\end{equation}
Observe that $g$ is defined and strictly increasing in the interval $(-\infty,\mu_1)$, and that it satisfies $g(0)=1$, $g(\be)\to1-n<0$ as $\be\to-\infty$. Thus there exists a unique $\overline{\be}=\overline{\be}(\mu_1,\ldots,\mu_n)<0$ such that $g(\overline{\be})=0$. Moreover, $g$ is defined and negative for $\be>\mu_n$. Consequently, the functions $\al_1,\ldots,\al_n$ are defined for $\be\in(\overline{\be},\mu_1)\cup(\mu_n,\infty)$.

\begin{proposition}\label{prop:locked}
Problem \eqref{eq:main} has a locked solution $(\be, u_1,\ldots,u_n)$ precisely in the following cases:
\renewcommand{\labelenumi}{(\roman{enumi})}
 \begin{enumerate}
 \item $\be\in (\overline{\be},\mu_1)\cup(\mu_n,\infty)$: Then
   $u_j=\al_j(\be)w, j=1,\ldots,n$, for some solution
   $w$ of \eqref{eq:scalar}.
\item $\be=\mu_1=\ldots=\mu_n$: Then $u_j=\al_jw$,
  $j=1,\ldots,n$, for some solution $w$ of \eqref{eq:scalar},
  $\al_1,\ldots, \al_n>0$, $\al^2_1+\ldots+\al^2_n=1/\be$.
 \end{enumerate}
\end{proposition}

\begin{proof}
If $(\be,u_1,\ldots, u_n)$ is a locked solution then $u_1$ solves
$-\Delta u_1+a(x)u_1 = cu^3_1$ for some constant $c>0$. It follows that $w=c^{1/2}u_1$ solves \eqref{eq:scalar}, and all $u_j$'s are multiples of $w$.

Now write $u_j=\al_jw$, $j=1,\ldots,n$, where $w$ is a solution of \eqref{eq:scalar}. Then \eqref{eq:main} leads to
\[
  \al_jw^3=(-\Delta+a(x))\al_j
  w=\left(\mu_j\al^3_j+\be\sum_{k\neq j}\al^2_k\al_j\right)w^3
\]
hence,
\begin{equation}\label{eq:alpha1}
  \mu_j\al^2_j+\be\sum_{k\neq
  j}\al^2_k=1,\quad j=1,\ldots, n.
\end{equation}
This implies
\begin{equation}\label{eq:alpha2}
  (\mu_i-\be)\al^2_i=(\mu_j-\be)\al^2_j\quad\text{for all }i,j.
\end{equation}
In case (i) it follows from \eqref{eq:alpha2} that $\al_j=\left(\frac{\mu_1-\be}{\mu_j-\be}\right)^{1/2}\al_1$ for $j=1,\ldots,n$. A simple computation using \eqref{eq:alpha1} now yields
\[
\al_1=\big((\mu_1-\be)g(\be)\big)^{-1/2}=\al_1(\be),
\]
and $\al_j=\al_j(\be)$ for all $j$ follows immediately. This proves in case (i) that a locked solution $(\be, u_1,\ldots,u_n)$ necessarily has the form $u_j=\al_j(\be)w$ for some
solution $w$ of \eqref{eq:scalar}. If (i) does not apply then we must have $\be=\mu_1=\ldots=\mu_n$, which implies $\be(\al^2_1+\ldots +\al^2_n)=1$, again by \eqref{eq:alpha1}. This is as claimed in (ii).

Finally, an elementary calculation shows that $(\be, u_1,\ldots, u_n)$ as described in (i) or (ii) is a solution of \eqref{eq:main}.
\end{proof}

Now we fix a solution $w$ of \eqref{eq:scalar} and set
\[
  u(\be):=(\al_1(\be)w,\ldots,\al_n(\be)w)
   \text{ for } \be\in(\overline{\be},\mu_1)\cup(\mu_n,\infty).
\]
We want to investigate the bifurcation of solutions of \eqref{eq:main} from
\[
\cT_w:=\left\{(\be,u(\be)):\be\in(\overline{\be},\mu_1)\cup(\mu_n,\infty)\right\}.
\]

%%%%%%%%%%%%%%%%%%%%%%%%%%%%%%%%%%%%%%%%%%%%%%%%%%%%%%%%%%%%%%%%%%%%%%%%%%%%%%%%%%%

\section{Necessary conditions for bifurcation}
\label{sec:bif-nec}

We need to linearize \eqref{eq:main} at $(\be,u(\be))\in\cT_w$. A simple calculation leads to the system
\begin{equation}
  \label{eq:main-lin}
  \left\{
    \begin{aligned}
     {}&(-\Delta+a(x))\phi=w^2 C(\be)\phi\\
{}&\phi=(\phi_1,\ldots, \phi_n)^\top\in E^n
    \end{aligned}
\right.
\end{equation}
with
\[
  C(\be)=E_n+\frac{2}{g(\be)} D(\be) \in \R^{n\times n}.
\]
Here $E_n$ is the $n\times n$ unit matrix and
\begin{equation*}\label{eq:def-ga}
  D(\beta)=\left(
    \begin{matrix}
      \mu_1\gamma^2_1&\beta\gamma_1\gamma_2&\ldots
      &\beta\gamma_1\gamma_n\\
\beta\gamma_1\gamma_2&\mu_2\gamma^2_2&\ldots&\beta\gamma_2\gamma_n\\
\dots &\dots & & \dots\\
\beta\gamma_1\gamma_n&\beta\gamma_2\gamma_n&\ldots& \mu_n\gamma^2_n.
    \end{matrix}
\right)
\end{equation*}
where $\ga_j=\ga_j(\be):=(\mu_j-\be)^{-1/2}$. Hence, $D(\beta)=(d_{ij}(\beta))_{i,j=1,\ldots,n}$ with
\begin{equation*}
d_{ij}(\beta)=\left\{
 \begin{aligned}
 &\beta\gamma_i(\beta)\gamma_j(\beta)
  = \frac{\be}{\big((\mu_i-\be)(\mu_j-\be)\big)^{1/2}}
  &&\quad\text{if }i\neq j;\\
 &\mu_i\gamma_i(\beta)^2 = \frac{\mu_i}{\mu_i-\be}
  &&\quad\text{if }i=j.
 \end{aligned}
\right.
\end{equation*}
Observe that $C(\be)$ and $D(\be)$ are symmetric matrices.

\begin{Lemma}\label{lem:c-eigen}
\begin{itemize}
\item[\rm a)] $C(\be)$ has the eigenvalues $3$ and
\[
f(\be):=1+\frac{2}{g(\be)} = 1+\frac{2}{1+\be\sum_{k=1}^n\frac{1}{\mu_k-\be}}.
\]
\item[\rm b)] If $\be\neq 0$ then the eigenvalue $3$ is simple with eigenvector
\begin{equation}\label{eq:first-ev}
b_1(\be)=\left(\ga_1(\be),\dots,\ga_n(\be)\right)^\top
=\left((\mu_1-\be)^{-1/2},\ldots,(\mu_n-\be)^{-1/2}\right)^\top,
\end{equation}
and the eigenvalue $f(\be)$ has multiplicity $n-1$ with eigenspace $b_1(\be)^\perp$. If $\be=0$ then $f(0)=3$ and $C(0)=3E_n$.
\end{itemize}
\end{Lemma}

\begin{proof}
A simple calculation shows that $D(\be)b_1(\be)=g(\be)b_1(\be)$, hence
$C(\be)b_1(\be)=3b_1(\be)$. Moreover, setting
\begin{equation}\label{eq:eigen}
b_j(\be)=(b_{j1},\ldots,b_{jn})^\top\quad
 \text{where } b_{jk}=
  \begin{cases}
   \ga_j(\be) &\text{if } k=1,\\
   -\ga_1(\be) &\text{if } k=j,\\
   0 &\text{else.}
  \end{cases}
\end{equation}
for $j=2,\ldots,n$ we have $b_j(\be)\perp b_1(\be)$ and $D(\be)b_j(\be) = b_j(\be)$. This implies $C(\be)b_j(\be)=f(\be)b_j(\be)$ for $j=2,\ldots, n$.

Since $f(\be)\neq 3$ for $\be\neq 0$ we see that $C(\be)$ has the eigenvalue $f(\be)$ with eigenspace $\opspan\{b_2,\ldots,b_n\}=b_1(\be)^\perp$. If $\be=0$ then $f(0)=3$, hence $3$ is the only eigenvalue of the symmetric matrix $C(0)$.
\end{proof}

\begin{remark}\label{rem:f}
Observe that $f:(\overline{\be},\mu_1)\to (1,\infty)$ is a strictly decreasing diffeomorphism. In fact,
\[
f'(\be) = \frac{-2g'(\be)}{g(\be)^2} < 0
 \quad\text{for }\overline{\be}<\be<\mu_1.
\]
Moreover, $f(\be)\to 1$ as $\be\nearrow\mu_1$, and $f(\be)\to\infty$ as
$\be\searrow\overline{\be}$.
\end{remark}

\begin{remark}\label{rem:T}
Since $C(\be)$ is a symmetric matrix depending smoothly on $\be$, and since the eigenvector $b_1(\be)$ also depends smoothly on $\be$, there exists a smooth map $T:(\overline{\be},\mu_1)\to SO(n)$ such that
\[
  T(\be)^{-1} C(\be)T(\be)=\mathrm{diag} (3,f(\be),\ldots, f(\be)).
\]
\end{remark}

A point $(\be,u(\be))\in \cT_w$ can be a bifurcation point of solutions of \eqref{eq:main} only if \eqref{eq:main-lin} has a nontrivial solution $\phi$. Nontrivial solutions of \eqref{eq:main-lin} are closely related to the weighted eigenvalue problem
\begin{equation}\label{eq:scalar-lin}
  -\Delta\psi+a(x)\psi=\lambda w^2\psi,\quad \psi\in E.
\end{equation}

\begin{remark}
Recall that $w$ is a non-degenerate solution of \eqref{eq:scalar} if and only if $\lambda=3$ is not an eigenvalue of \eqref{eq:scalar-lin}. In the degenerate case where $\psi \neq 0$ is a solution of \eqref{eq:scalar-lin} with $\lambda=3$, we see that
$\phi=b_1(\be)\psi\in E^n$, $b_1(\be)\in\R^n$ as in \eqref{eq:first-ev}, is a solution \eqref{eq:main-lin} for every $\overline{\be}<\be<\mu_1$. If $w$ is not an isolated solution of \eqref{eq:scalar}, say $w_k\to w$ for a sequence of solutions of \eqref{eq:scalar}, then we have branches $\cT_{w_k}$ of solutions of \eqref{eq:main}. In this case every point on
$\cT_w$ is a bifurcation point for \eqref{eq:main}. The problem is more subtle if $w$ is an isolated, degenerate solution of \eqref{eq:scalar}. This case will not be treated here.
\end{remark}

From now on we assume that $w$ is a non-degenerate solution of \eqref{eq:scalar}. Then \eqref{eq:scalar-lin} has a sequence of eigenvalues $\lambda_1=1<\lambda_2<\lambda_3<\ldots$ such that $\lambda_k\neq 3$ for all $k\in\N$ and $\lambda_k\to\infty$ as $k\to\infty$. The multiplicity of $\lambda_k$ as eigenvalue of \eqref{eq:scalar-lin} is denoted by
\begin{equation}
\label{eq:kernel-scalar}
n_k=\dim V_k \quad\text{with } V_k=\opkern(-\Delta+a-\lambda_kw^2).
\end{equation}

\begin{proposition}\label{prop:bif-nec}
\begin{itemize}
\item[\rm a)] A point $(\be, u(\be))\in\cT_w$ is a bifurcation point only if $\be=\be_k:=f^{-1}(\lambda_k)$ for some $k\geq 2$ with $f$ from Lemma~\ref{lem:c-eigen}.
\item[\rm b)] If $\be=\be_k$ then $\phi=(\phi_1,\ldots, \phi_n)\in E^n$ solves \eqref{eq:main-lin} if, and only if, $\phi_j\in V_k$ for $j=1,\ldots, n$, and $\phi\perp b_1(\be)$ with $b_1(\be)\in\R^n$ from \eqref{eq:first-ev}, that is, $\sum^n_{j=1}\ga_j(\be)\phi_j=0$. In particular, the space of solutions of \eqref{eq:main-lin} has dimension $(n-1)\cdot n_k$.
\end{itemize}
\end{proposition}

\begin{proof}
a) Let $T(\be)\in SO(n)$ be as in Remark \ref{rem:T}. Then $\phi$ solves \eqref{eq:main-lin} if, and only if, $\psi=T(\be)^{-1}\phi$ solves
\begin{equation}
  \label{eq:lin-transf}
  \left\{
  \begin{aligned}
    -\Delta\psi_1+a\psi_1&=3w^2\psi_1\\
-\Delta\psi_j+a\psi_j&=f(\be)w^2\psi_j,\quad j=2,\ldots, n.
  \end{aligned}
\right.
\end{equation}
Since $w$ is non-degenerate this implies $\psi_1=0$, hence a nontrivial solution $\psi$ exists, if and only if $f(\be)=\lambda_k$ for some $k\geq 2$. (The case $k=1$, i. e. $\lambda_1=1$, $\be=\mu_1$ does not apply; see also Remark \ref{rem:mu1}.) If $f(\be)=\lambda_k$ then $\psi=(\psi_1,\ldots, \psi_n)^\top$ with
$\psi_1=0$, $\psi_2,\ldots, \psi_n\in V_k$, solves \eqref{eq:lin-transf}.

b) We fix $\be=f^{-1}(\lambda_k)$ and write $T=T(\be)=(T_1,\ldots,T_n)$. The transformations $\psi\mapsto T\psi$, $\phi\mapsto T^{-1}\phi$, are isomorphisms between the solutions of
\eqref{eq:lin-transf} and \eqref{eq:main-lin}. As a consequence of Lemma \ref{lem:c-eigen} we see that $T_1$ is a multiple of $b_1(\be)$, and
$\opspan\{T_2,\dots,T_n\} = b_1(\be)^\perp$. Thus $\phi$ solves \eqref{eq:main-lin} if, and only if, $\psi= T^{-1}\phi=T^\top\phi$ solves
\begin{equation}\label{eq:eigen1}
  \langle T_1,\phi\rangle=\psi_1=0,
  \text{ hence }0=\langle b_1(\be),\phi\rangle=\ga_1(\be)\phi_1+\ldots +\ga_n(\be)\phi_n
\end{equation}
and
\begin{equation}\label{eq:eigen2}
  \langle
  T_j,\phi\rangle=\psi_j\in V_k,\text
  { for }j=2,\ldots, n.
\end{equation}
Recall that the eigenvectors $b_2,\ldots, b_n$ from \eqref{eq:eigen} satisfy $\opspan\{b_2, \ldots, b_n\}=\opspan\{T_2,\ldots, T_n\}$. It follows that \eqref{eq:eigen2} is equivalent to
\begin{equation}\label{eq:eigen3}
  \langle b_j,\phi\rangle=\ga_j(\be)\phi_1-\ga_1(\be)\phi_j\in V_k
	\quad\text{for } j=2,\ldots, n.
\end{equation}
Now \eqref{eq:eigen1} and \eqref{eq:eigen3} are equivalent to
\[
 \phi_1,\ldots, \phi_n\in V_k\quad\text{ and }\quad
 \sum^n_{j=1}\ga_j(\be)\phi_j=0.
\]
\end{proof}

\begin{remark}\label{rem:mu1}
By our definition $\cT_w$ does not contain a point with the parameter value $\be=\mu_1$. If $\mu_1<\mu_n$ then it is not difficult to show that $(\mu_1,u_1,\ldots,u_n)\in\overline{\cT}_w$ implies that at least one $u_j=0$. If $\mu_1=\mu_n$ then $(\mu_1,(n\mu_1)^{-1/2}w,\ldots,(n\mu_1)^{-1/2}w)\in\overline{\cT}_w$. This is a bifurcation point where the bifurcating locked solutions are described in Proposition \ref{prop:locked}:
\[
  (\mu_1,\al_1w,\ldots, \al_nw)\text{ with }
  \al_j>0,\quad \sum^n_{j=1}\al_j^2=1/\be=1/\mu_1.
\]
\end{remark}

%%%%%%%%%%%%%%%%%%%%%%%%%%%%%%%%%%%%%%%%%%%%%%%%%%%%%%%%%%%%%%%%%%%%%%%%%%%%%%%%%%%

\section{Sufficient conditions for bifurcation}
\label{sec:bif-suf}

We first recall the variational structure of problem \eqref{eq:main}. Since $N\leq 3$ the functional
\begin{equation*}
\begin{aligned}
&J_\be(u_1,\ldots, u_n)\\
&\hspace*{1cm}
 =\frac{1}{2}\sum^n_{j=1}\int_\Om(|\nabla u_j|^2+a(x)u^2_j)
  -\frac{1}{4}\sum^n_{j=1}\int_\Om\mu_j u^4_j
  - \frac{\be}{2}\sum_{i< j}\int_\Om u^2_i u^2_j\\
&\hspace*{1cm}
 =\frac{1}{2}\sum^n_{j=1}\|u_j\|^2_E
  -\frac{1}{4}\sum^n_{j=1}\mu_j\|u_j\|^4_{L^4}
  -\frac{\be}{2}\sum_{i< j}\int_\Om u^2_i u^2_j
\end{aligned}
\end{equation*}
is well defined for $u_1,\ldots, u_n\in E$. It is well known that $J_\be:E^n\to\R$ is of class $C^2$, and critical points of $J_\be$ are (weak) solutions of \eqref{eq:main}.

For $\be\in(\overline{\be},\mu_1)\cup(\mu_n,\infty)$ let $m(\be)$ be the Morse index of $u(\be)=(\al_1(\be)w,\ldots,\al_n(\be)w)$ as critical point of $J_\be$. Recall the bifurcation parameters $\be_k$ from Proposition \ref{prop:bif-nec}.

\begin{Lemma}\label{lem:bif-suff}
For $k\geq 2$ there holds
\[
  m(\be_k-\eps)-m(\be_k+\eps)=(n-1)n_k\qquad\text{ for }\eps>0\text{ small.}
\]
\end{Lemma}

\begin{proof}
Consider the quadratic form on $E^n$ given by
\begin{equation*}
\begin{aligned}
Q_\be(\phi)
 &=\sum^n_{j=1}\int_\Om(|\nabla\phi_j|^2+a\phi^2_j)
    -\int_\Om w^2\langle C(\be)\phi,\phi\rangle\\
 &=\|\phi\|^2_{E^n}-\int_\Om w^2\langle C(\be)\phi,\phi\rangle\\
\end{aligned}
\end{equation*}
where $C(\be)$ is as in \eqref{eq:main-lin}. Then $m(\be)$ is precisely the index of $Q_\be$. Let $E^n=V^+_{\be_k}\oplus V^0_{\be_k}\oplus V^-_{\be_k}$ be the decomposition of $E^n$ into the positive eigenspace $V^+_{\be_k}$ of $Q_{\be_k}$, the negative eigenspace $V^-_{\be_k}$ of $Q_{\be_k}$, and the kernel
\begin{equation}\label{eq:kernel}
 V^0_{\be_k}=\left\{\phi\in E^n:\phi_j\in V_k \text{ for }
              j=1,\ldots,n,\ \sum^n_{j=1}\ga_j(\be)\phi_j=0\right\}.
\end{equation}
The derivative $Q_\be'=\frac{\pa}{\pa\be}Q_\be$ of $Q_\be$ with respect to $\be$ is given by
\[
  Q'_\be(\phi)=-\int_\Om w^2\langle C'(\be) \phi,\phi\rangle.
\]
Recall that $C(\be)$ is a smooth function of $\be$. Now we have
\[
Q_\be=Q_{\be_k}+(\be-\be_k)Q'_{\be_k}+o(|\be-\be_k|)\text{ as } \be\to \be_k.
\]
This implies that $Q_\be>0$ on $V^+_{\be_k}$ and $Q_\be<0$ on $V^-_{\be_k}$, if $\be$ is close to $\be_k$. Thus the lemma follows if $Q'_{\be_k}$ is positive on $V^0_{\be_k}$. In order to prove this let $T(\be)\in SO(n)$ be as in \ref{rem:T}, that is, $T$ depends smoothly on $\be$, and satisfies
\[
  T(\be)^{-1}C(\be)T(\be)=\diag (3,f(\be),\ldots, f(\be))=:C_T(\be).
\]
It follows that $C(\be)T(\be)=T(\be) C_T(\be)$, hence
\begin{equation*}
C'(\be)
 =T'(\be)C_T(\be)T(\be)^{-1}+T(\be)C'_T (\be)T(\be)^{-1}-C(\be)T'(\be)T(\be)^{-1}
\end{equation*}
For an eigenvector $u\in b_1(\be)^\perp\subset\R^n$ of $C(\be)$ corresponding to the eigenvalue $f(\be)$ we compute
\[
T'(\be)C_T(\be)T(\be)^{-1}u=f(\be)T'(\be)T(\be)^{-1}u,
\]
and
\[
T(\be)C'_T(\be)T(\be)^{-1} u=f'(\be) u,
\]
and
\[
\begin{aligned}
\langle C(\be)T'(\be)T(\be)^{-1}u,u\rangle
 &=\langle T'(\be)T(\be)^{-1}u,C(\be)u\rangle\\
 &=f(\be)\langle T'(\be)T(\be)^{-1}u,u\rangle.
\end{aligned}
\]
This implies
\[
\langle C'(\be)u,u\rangle = f'(\be)|u|^2.
\]

According to Proposition \ref{prop:bif-nec} a function $\phi\in V^0_{\be_k}$ satisfies
\[
\sum^n_{j=1}\ga_j(\be_k)\phi_j=0, \text{ i.~e.\ }
  \phi(x)\in b_1(\be_k)^\perp\subset\R^n\text{ for every }x\in\Om.
\]
Here we used Proposition \ref{prop:bif-nec} b). Consequently, for $\phi\in V^0_{\be_k}$ we have
\[
Q'_{\be_k}(\phi)
 =-\int_\Om w^2\langle C'(\be_k)\phi,\phi\rangle
 =-f'(\be_k)\int_\Om |\phi|^2 w^2
 >0
\]
because $f'(\be_k)<0$ by Remark \ref{rem:f}. Therefore $Q'_{\be_k}$ is positive definite on $V_{\be_k}^0$, as required.
\end{proof}

\begin{Theorem}
  For every $k\geq 2$, the point $(\be_k, u(\be_k))\in\cT_w$ is a bifurcation point of solutions of \eqref{eq:main}.
\end{Theorem}

\begin{proof}
We can apply a classical bifurcation theorem of Krasnoselski \cite{Krasnoselski:1964} in the version of \cite[Theorem 8.9]{Mawhin-Willem:1989}. Care has to be taken in order to show that the bifurcating solutions are actually positive. For this one may argue as in \cite[pp.~354-355]{bartsch-dancer-wang:2010}. Alternatively one can use \cite[Corollary~1.4,~Theorem~1.5]{Bartsch-Dancer:2010} in order to prove the existence of solutions bifurcating into the positive cone of $E^n$. Both arguments yield that the bifurcating solutions $(\be,u_1,\dots,u_n)$ satisfy $u_j>0$ for $j=1,\dots,n$.
\end{proof}

\begin{remark}
a)  If $(n-1)n_k$ is odd then there exists a connected set
$\cS_k\subset\R\times E^n\setminus\cT_w$ of solutions $(\be, u)$ of \eqref{eq:main}
such that $(\be_k,u(\be_k))\in\overline{\mathcal{S}}_k$. This can be proved using
degree theory in the spirit of Rabinowitz' global bifurcation theorem. Whether
$\mathcal{S}_k$ is unbounded in the $\be$-direction, or in the $E^n$-direction, or
whether it returns to $\mathcal{T}_w$ is unclear in general.

b)  The case $n=2$ has been treated in \cite{bartsch-dancer-wang:2010}.
Clearly, if $n$ is odd degree theoretic methods do not apply. On the other hand, one
may hope for an $(n-1)n_k$-dimensional manifold of bifurcating solutions, or multiple
bifurcating branches. In many cases, a high dimensional kernel is not generic but
forced by some symmetry. In the next section we present some hidden symmetry, which
is not inherent to the full functional, but to a special type of solutions. In this
way we obtain multiple bifurcating branches.
\end{remark}

%%%%%%%%%%%%%%%%%%%%%%%%%%%%%%%%%%%%%%%%%%%%%%%%%%%%%%%%%%%%%%%%%%%%%%%%%%%%%%%%%%%

\section{Branches of partially locked solutions}
\label{sec:ampli-sol}

In this section we investigate partially locked solutions of \eqref{eq:main},
i.~e.\ solutions $(\be,u)$ where $u_i/u_j$ is constant for some, but not all
$i\neq j$. If $u_i/u_j$ is constant we say that $u_i$ and $u_j$ are locked.

\begin{Lemma}\label{lem:locked1}
  If $(\be,u)$ solves \eqref{eq:main} and $u_i,
  u_j$ are locked then
\[
u_j = \left(\frac{\mu_i-\be}{\mu_j-\be}\right)^{1/2} u_i
 = \frac{\ga_j(\be)}{\ga_i(\be)}\ u_i\ .
\]
\end{Lemma}

\begin{proof}
We set $u_j=\ga u_i$ for some $\ga>0$. Then the
equations \eqref{eq:main} for $u_i$ and $u_j$ lead to
\[
 -\Delta u_i+au_i=(\mu_i+\be\ga^2)u^3_i+\be\sum_{k\ne i,j}u^2_k u_i
\]
and, respectively, to
\[
  -\Delta u_i+au_i=(\mu_j\ga^2+\be) u^3_i+\be\sum^n_{k\ne i,j}u^2_ku_i
\]
It follows that
\[
\ga = \left(\frac{\mu_i-\be}{\mu_j-\be}\right)^{1/2}
 = \frac{\ga_j(\be)}{\ga_i(\be)}.
\]
\end{proof}

The next result reveals some hidden symmetry in the problem which will
allow a dimension reduction argument below. It is simple to prove but
has not been observed before.

\begin{Lemma}\label{lem:locked2}
Suppose $(\be, u)$ solves \eqref{eq:main} for some $i\in\{1,\ldots,n\}$,
and suppose
\[
u_j = \left(\frac{\mu_i-\be}{\mu_j-\be}\right)^{1/2} u_i
 = \frac{\ga_j(\be)}{\ga_i(\be)}\ u_i\ .
\]
Then \eqref{eq:main} also holds for $u_j$.
\end{Lemma}

\begin{proof}
We set $\ga:=\frac{\ga_j(\be)}{\ga_i(\be)}$, so $u_j=\ga u_i$.
Multiplying the equation for $u_i$
\[
-\Delta u_i+au_i
 =\mu_i\mu^3_i+\be\sum_{k\ne i}u^2_k  u_i
 =(\mu_i+\be\ga^2)u^3_i+\be\sum_{k\ne i,j}u^2_k u_i
\]
by $\ga$ yields
\begin{equation*}
\begin{aligned}
-\Delta u_j+au_j
 &=(\mu_i+\be\ga^2)\ga u^3_i+\be\sum_{k\ne i,j} u^2_k u_j
  =(\mu_j\ga^2+\be)\ga u^3_i+\be\sum_{k\ne i,j} u^2_k u_j\\
{}&=\mu_j u^3_j+\be\sum_{k\neq j}u^2_ku_j.
\end{aligned}
\end{equation*}
This is equation \eqref{eq:main} for $u_j$ as claimed.
\end{proof}

Let $\cP=\{P_1,\ldots,P_m\}$ be a partition of $\{1,\ldots, n\}$,
i.~e.\ $\{1,\ldots, n\}=P_1\dot\cup\ldots\dot\cup P_m$ is a disjoint union and
$P_1,\ldots, P_m\neq\emptyset$. We write $|\cP|:=m$ for the cardinality
of $\cP$. For $\be<\mu_1$ we define
\[
 X^\cP_\be:=\left\{u\in E^n\mid \forall k=1,\ldots, m\ \forall\  i,j\in\cP_k:
               u_j=\frac{\ga_j(\be)}{\ga_i(\be)}u_i\right\}.
\]
A solution $u\in X^\cP_\be$ of \eqref{eq:main} will be called
$\cP$-locked.

\begin{proposition}\label{prop:reduction}
If $u\in X^\cP_\be$ is a critical point of $J_\be|X^\cP_\be$ then $u$ is
a critical point of $J_\be$, hence a solution of \eqref{eq:main}.
\end{proposition}

\begin{proof}
We may assume $u_i\in P_i$ for $i=1,\ldots, m$. That $u\in X^\cP_\be$ is a
critical point of $J_\be|X^\cP_\be$ is equivalent to $u_1,\ldots, u_m$ being
solutions of \eqref{eq:main} with $u_j=\frac{\ga_j(\be)}{\ga_i(\be)}u_i$
$j\in P_i$, $i=1,\dots,m$. Lemma \ref{lem:locked2} implies that
$u=(u_1,\ldots,u_n)$ solves \eqref{eq:main}, hence $u$ is a critical point
of $J_\be$.
\end{proof}

One can now use Proposition \ref{prop:reduction} to find bifurcations
of partially locked solutions in $X^\cP_\be$. We may assume that
$i\in P_i$ for $i=1,\ldots, m$. Define
\[
 i^\cP_\be:E^m\to X^\cP_\be,\quad (u_1,\ldots,u_m)\mapsto(u_1,\ldots, u_n),
\]
by
\[
  u_j=\frac{\ga_j(\be)}{\ga_i(\be)} u_i \quad
  \text{if } j\in P_i,\ i=1,\ldots, m.
\]
Let $\pi: E^n\to E^m$ be the projection onto the first $m$ components,
so that $\pi\circ i_\be^\cP$ is the identity on $E^m$ for every
$\be<\mu_1$. Finally we define
\[
  J^\cP_\be:= J\circ i^\cP_\be : E^m\to\R
\]
for $\be\in\R$. A critical point $u\in E^m$ of $J^\cP_\be$ yields a critical
point $i^\cP_\be(u)\in X^\cP_\be$ of $J_\be|X^\cP_\be$, hence a $\cP$-locked
solution of \eqref{eq:main} as a consequence of
Proposition \ref{prop:reduction}. The branch
\[
\cT^\cP_w=\{(\be,\pi u):(\be,u)\in\cT_w\}=\{(\be, u):(\be,i^\cP_\be(u))\in\cT_w\}
\]
is a branch of critical points corresponding to the branch $\cT_w$ of
locked solutions. \\

Let $m^\cP(\be)$ be the Morse index of
$\pi u(\be)=(\al_1(\be)w,\ldots,\al_m(\be)w)$
as critical point of $J^\cP_\be$. Bifurcation of
$\cP$-locked solutions occurs if $m^\cP(\be)$ changes.

\begin{Lemma}\label{lem:bif-locked-suff}
For $k\geq 2$ there holds
\[
 m^\cP(\be_k-\epsilon)-m^\cP(\be_k+\epsilon)=(|\cP|-1)n_k
\qquad\text{for }\eps>0\text{ small.}
\]
\end{Lemma}

\begin{proof}
This follows from Proposition \ref{prop:reduction} and Lemma
\ref{lem:bif-suff} if we can show that
\[
\dim (V^0_{\be_k}\cap X^\cP_{\be_k})=(|\cP|-1)n_k ;
\]
here $V^0_{\be_k}$ is the kernel of the Hessian of $J_{\be_k}$ at $u(\be_k)$ as in
\eqref{eq:kernel}, and $n_k=\dim V_k$ as in \eqref{eq:kernel-scalar}; recall
the description of $V^0_{\be_k}$ in Proposition \ref{prop:bif-nec}.
For simplicity we write $\ga_j=\ga_j(\be_k)$ below.
\begin{equation*}
\begin{aligned}
\phi\in V^0_{\be_k}\cap X^\cP_{\be_k}
 \quad\Longleftrightarrow\quad &\ \phi_1,\ldots, \phi_n\in V_k,\
     \sum^n_{j=1}\ga_j\phi_j=0,\text{ and}\\
   {}&\ \phi_j=\frac{\ga_j}{\ga_i}\phi_i\text{ for } j\in P_i,\ i=1,\ldots,m=|\cP|\\
 \quad\Longleftrightarrow\quad &\ \phi_1,\ldots, \phi_m\in V_k,\
     \sum^m_{i=1}\left(\sum_{j\in P_i}\frac{\ga^2_j}{\ga_i}\right)\phi_i=0,\text{ and}\\
   {}&\ \phi=i^\cP_{\be_k}(\phi_1, \ldots,\phi_m)
\end{aligned}
\end{equation*}
\end{proof}

\begin{Corollary}\label{cor:bif-locked}
For every partition $\cP$ of $\{1,\ldots,n\}$ with $|\cP|\geq 2$, and for every
$k\geq2$ the point $(\be_k,u(\be_k))\in\cT_w$ is a bifurcation point of $\cP$-locked
solutions of \eqref{eq:main}. It is a global bifurcation point of $\cP$-locked
solutions if $(|\cP|-1)\cdot n_k$ is odd.
\end{Corollary}

In the application of the Corollary \ref{cor:bif-locked} one has to be cautious because a $\cP$-locked solution may also be $\cP '$-locked for $\cP\neq\cP '$. This is the case if for any $P'_i\in\cP '$ there exists a $P_j\in\cP$ with $P'_i\subset P_j$. On the other hand, the bifurcating $\cP$-locked solutions are different from the $\cP'$-locked solutions if every solution which is both $\cP$-locked and $\cP'$-locked is automatically completely locked. Recall that the completely locked solutions are those in $\cT_w$. \\

For a subset $A\subset\{1,\ldots,n\}$ we set $A^c:=\{1,\ldots,n\}\setminus A$ and $\cP_A:=\{A, A^c\}$. Given two subsets $\emptyset\neq A,B\subsetneq\{1,\ldots, n\}$ such that $A\neq B$ and $A^c\neq B$, a solution which is $\cP_A$-locked and $\cP_B$-locked is
necessarily completely locked. We therefore obtain the following multiplicity result.

\begin{Corollary}\label{cor:bif-mult}
Suppose $n_k$ is odd. For every subset $\emptyset\neq A\subsetneq\{1,\ldots, n\}$ there exists a global branch $\cS^A_k$ of $\cP_A$-locked solutions of \eqref{eq:main} bifurcating from $\cT_w$ at $(\be_k,u(\be_k))$. The branches $\cS^A_k$ and $\cS^B_k$ are disjoint except when $A=B$ or $A=B^c$. In particular, there exist at least $2^{n-1}-1$ such global branches which are different.
\end{Corollary}

\begin{remark}
a) The corollary applies, for instance, if $\Om\subset\R$ is a bounded interval, or if $\Om\subset\R^N$, $2\le N\le3$, is radially symmetric and one looks for radially symmetric solutions, i.~e.\ $G=O(N)$. Then problem \eqref{eq:scalar} has a unique solution which is also nondegenerate in $E$. This has been proved in \cite{ni-takagi:1991, ni-wei:1995, Tang, FMT}, in the case of a ball, an annulus, and $\R^N$.

b) If $n_k=1$ then one can apply the Crandall-Rabinowitz theorem to show that the bifurcating branches $\cS^A_k$ of $\cP_A$-locked solutions of \eqref{eq:main} are $C^1$-curves near the bifurcation point.

c) The global features of the branches $\cS_k^A$ are not known. If $\Omega$ is a ball or an annulus and $n=2$ it has been proved in \cite{bartsch-dancer-wang:2010} that the bifurcating branches $\cS_k=\cS_k^{\{1\}}\subset(-\infty,0)\times E^2$ are bounded over bounded subsets of $(-\infty,0)$, hence they cover all of $(-\infty,\be_k)$. An extension of this result to $n>2$ is work in progress.
\end{remark}

{\bf Acknowledgment:} The author thanks Pavol Quittner for discussions on the topic, and for the invitation and hospitality during a visit at the Comenius University in Bratislava in spring 2011.

%%%%%%%%%%%%%%%%%%%%%%%%%%%%%%%%%%%%%%%%%%%%%%%%%%%%%%%%%%%%%%%%%%%%%%%%%%%%%%%%%%%

%--------------------------- DEFINITIONS for references -----
\def\by{\relax}
\def\paper{\relax}
\def\jour{\textit}
\def\vol{\textbf}
\def\yr#1{\rm(#1)}
\def\pages{\relax}
\def\book{\textit}
\def\inbook{In: \textit}
\def\finalinfo{\rm}
%----------------------------

\bibliographystyle{amsplain}

\noindent
{\sc Address of the author:}\\[1em]
Thomas Bartsch\\
Mathematisches Institut\\
University of Giessen\\
Arndtstr.\ 2\\
35392 Giessen\\
Germany\\
Thomas.Bartsch@math.uni-giessen.de

\end{document}